\documentclass{article}
\usepackage{amsfonts}
\usepackage{amsmath}
\usepackage{amsthm}
\usepackage{amssymb}
\usepackage{comment}
\usepackage[latin1]{inputenc}
\usepackage{eurosym}
\usepackage{graphicx}
\usepackage{epsfig}
\usepackage{hyperref}
\usepackage{dsfont}
\usepackage{color}
\usepackage{xcolor}
\usepackage{appendix}

\usepackage[displaymath,mathlines]{lineno}

\allowdisplaybreaks

\DeclareMathOperator*{\argmax}{arg\,max}

\usepackage{mathtools}
\usepackage{verbatim}
\usepackage{ifthen}
\hypersetup{urlcolor=blue, citecolor=red}
\usepackage{hyperref}
\newtheorem{theorem}{Theorem}[section]
\newtheorem{corollary}[theorem]{Corollary}
\newtheorem{lemma}[theorem]{Lemma}
\newtheorem{proposition}[theorem]{Proposition}

\theoremstyle{definition}
\newtheorem{definition}[theorem]{Definition}

\newtheorem{assumption}[theorem]{Assumption}

\newcommand{\Rd}{\mathbb{R}^d}
\newcommand{\R}{\mathbb{R}}

\title{Benamou-Brenier Formulation of Optimal Transport for Nonlinear Control Systems on $\Rd$} 
\author{Karthik Elamvazhuthi}

\begin{document}
\date{}
\maketitle
\begin{center}
\textit{Applied Mathematics and Plasma Physics Group \\Los Alamos National Laboratory \\ Los Alamos, USA} 
\end{center}

\bigskip

\begin{abstract}
In this paper we consider the Benamou-Brenier formulation of optimal transport for nonlinear control affine systems on $\Rd$, removing the compactness assumption of the underlying manifold in previous work by the author. By using Bernard's Young measure based weak formulation of optimal transport, the results are established for cases not covered by previous treatments using the Monge problem. Particularly, no assumptions are made on the non-existence of singular minimizing controls or the cost function being Lipschitz. Therefore, the existence of solutions to dynamical formulation is established for general Sub-Riemmanian energy costs not covered by literature previously. 

The results also establish controllability of the continuity equation whenever the corresponding Kantorovich problem admits a feasible solution, leveraging the equivalence between the Kantorovich and Benamou-Brenier formulations. Furthermore, when the cost function does not admit singular minimizing curves, we demonstrate that the Benamou-Brenier problem is equivalent to its convexified formulation in momentum and measure coordinates. In this regular setting, we further show that the constructed transport solutions possess sufficient regularity: for the feedback control laws that achieve transport, the associated continuity equation admits a unique weak solution. These findings apply in particular to linear-quadratic costs for controllable linear time-invariant (LTI) systems, as well as to certain classes of driftless nonlinear systems. Thus, in these cases, controllability of the continuity equation is achieved with control laws regular enough to guarantee uniqueness of solutions.
\end{abstract}


\section{Introduction}\label{section2}

In this paper, we are interested in optimal transport problems associated with costs arising from control problems. Towards this end, let $c:\Rd \times \Rd \rightarrow \R \cup \{\infty\}$ be a cost function. The optimal transport problem has two classical {\it static formulations}, known as the {\it Monge problem} and the {\it Kantorovich problem} \cite{villani2003topics}. Given probablity measures $\mu_0$ and $\mu_T$, the Monge problem is the following,
 \begin{eqnarray}
\label{eq:mong}
C_{\rm mon}(\mu_0,\mu_T) :=\inf_{T:\Rd \rightarrow \Rd} \int_{x \in \mathbb{R}^d} c(x,T(x))dx \\
s.t. ~~ T_{\#} \mu_0 = \mu_T
\end{eqnarray}
 where $T_{\#}\mu$ is the measure defined by
	\begin{equation}
	(T_\# \mu)(\Omega) = \mu(T^{-1}(\Omega)), 
	\end{equation}
	for all Borel measurable sets $\Omega \subseteq \mathbb{R}^d$. 

 Additionally, one has the Kantorovich formulation, 

 \begin{equation}
 \label{eq:kant}
C_{\rm kan}(\mu_0,\mu_T) =\inf_{K \in \Pi} \int_{\Rd \times \Rd} c(x,y) d \gamma(x,y)
 \end{equation}
 where $\Pi = \{\eta \in \mathcal{P}(\Rd \times \Rd) = \pi^1_{\#} \eta = \mu_0,~ \pi^2_{\#} \eta = \mu_T\}$ and $\pi^1 :\mathbb{R}^d \times \Rd \rightarrow \Rd$, $\pi^2 :\mathbb{R}^d \times \Rd \rightarrow \Rd$ are the projection maps into the first and second coordinate, respectively.
 
To consider costs $c(x,y)$ that arise from optimal control problems, we need some additional definitions. Suppose $f_i:\mathbb{R}^d \rightarrow \mathbb{R}^d$ are smooth vector fields for $i =0,...,n$, with the interesting case being when $n <d$. Consider the following finite-dimensional control system on $\mathbb{R}^d$,
	\begin{equation}
	    \dot{\omega}(t) = f_0(\omega(t))+ \sum_{i=1}^n u_i(t) f_i(\omega(t))
	    \label{eq:ctrsyspre}
	\end{equation}
	where $\omega(t)$ represents the state and $u(t):=[u_1(t),...u_n(t)]^T$ are the control inputs of the system.
Control systems of this type are said to be in control-affine form and are well-studied in control theory literature \cite{agrachev2013control}. 

Let $T>0$. Given $x,y \in \mathbb{R}^d$, a standard instance of the optimal control problem for this system is to solve the following {\it fixed end-point  Lagrange} problem,
\begin{equation}
\label{eq:OCP}
c(x,y)=
\inf_{\omega, u} J(\omega,u) := \int_0^T\ L(t,\omega(t),u(t))dt 
\end{equation}
subject to \eqref{eq:ctrsyspre} and the constraints 
\begin{equation}
\omega(0) =x ~~\omega(T)=y 
\end{equation}
where $L :[0,T] \times \Rd \times \R^n \rightarrow \R$ is the running cost.

The optimal transport problem for these classes of costs has been addressed in \cite{agrachev2009optimal,hindawi2011mass,rifford2014sub,chen2016optimal} under various class of assumptions on the dynamics \eqref{eq:ctrsyspre} and regularity. The paper \cite{agrachev2009optimal} established existence of solutions to the Monge problem when the control system is driftless and does not admit sharp minimizers. The result also applied to the case of linear quadratic costs for linear time invariant systems, which was was addressed in more detail in \cite{hindawi2011mass}. The result in \cite{figalli2010mass} proved that the Monge solution for minimum energy control of driftless systems admits a unique minimizer even when only singular minimizers are absent, a condition weaker than the one assumed in \cite{agrachev2009optimal}. 

The optimal transport problem that we are interested in is a variation of the above problem where the initial and final condition of the state $\omega(t)$ are represented by probability densities $\rho_0:\mathbb{R}^d \rightarrow \mathbb{R}$ and $\rho_T:\mathbb{R}^d \rightarrow \mathbb{R}$. 
Suppose that the controls $u$ are given in {\it feedback form}, using functions $u(\omega(t),t)$ of the variable $\omega(t)$. Then given by the following system of equations,
\begin{equation}\label{eq:FP}
\begin{cases}
\partial_t \rho+ \nabla \cdot \left(  f_0\rho
\right) +  \sum_{i=1}^n \nabla \cdot \left( u_i f_i \rho \right) =0, ~ &\text{in}~ \mathbb{R}^d \times [0,T]\\
\rho(0,\cdot)= \rho_0, \quad \rho(T,\cdot)= \rho_T ~ &\text{in}~  \mathbb{R}^d.
\end{cases}
\end{equation}
Subject to these constraints, the optimal transport problem that we wish to solve is the following optimization problem,
\begin{equation}\label{eq:main_optim}
\begin{split}
\inf_{u,\rho} \int_0^T\int_{\mathbb{R}^d}   L(t,x,u(t,x)) \rho(t,x)dx dt 
\end{split}
\end{equation}
The advantage of this continuum formulation, over the Monge-Kantorovich formulation of the optimal transport problem, as demonstrated by Benamou-Brenier \cite{benamou2000computational} for the single integrator case ($\dot{\omega}(t) = u$), is that for certain running costs $L(t,x,u)$ the problem can be convexified using a suitable change of variables. For example, if $L (t,x,u)  =|u|^2$. Then considering the variables $(m,\rho)$, where $m= u \rho$, leads to a convex optimization problem,
\textcolor{black}{
\begin{equation}
\label{eq:treq}
\inf_{m,\rho} \int_0^T\int_{\mathbb{R}^d} \frac{1}{2} \frac{|m(t,x)|^2}{\rho(t,x)} dx dt 
\end{equation}}
subject to the linear constraints,
\begin{equation}
\label{eq:ctqcnt}
\begin{cases}
\partial_t \rho+ \nabla \cdot \left(  f_0\rho
\right) +  \sum_{i=1}^n \nabla \cdot \left( f_i m_i \right) =0, ~ &\text{in}~ \mathbb{R}^d \times [0,T]\\
\rho(0,\cdot)= \rho_0, \quad \rho(T,\cdot)= \rho_T ~ &\text{in}~  \mathbb{R}^d.
\end{cases}
\end{equation}

In \cite{chen2016optimal}, the authors addressed this  dynamical version of the optimal transport problem for minimum energy costs associated with linear control systems. In \cite{caluya2020finite}, this result has been leveraged to consider dynamical transport of nonlinear systems that are feedback linearizable. In \cite{elamvazhuthi2023dynamical}, the case of general nonlinear systems on compact manifolds was addressed.  

Another related problem, the controllability of the continuity equation \eqref{eq:ctqcnt} with control $u(t,x)$ has been an important area of research \cite{brockett2012notes,duprez2019approximate,raginsky2024some}. When the control system is nonlinear and driftless, for positive initial and final distributions, one can conclude controllability either using properties of the corresponding sub-Laplacian \cite{khesin2009nonholonomic,arguillere2017sub,elamvazhuthi2024linear} or by posing controllability problem in terms of the corresponding diffeomorphisms \cite{agrachev2009controllability}.

The goal of this paper is to revisit the Benamou-Brenier formulation of the optimal transport problem for general nonlinear control systems and to explore issues of controllability of the continuity equation through the lens of optimal transport.

We begin by considering the state space to be $\mathbb{R}^d$, thereby removing the compactness assumption imposed in \cite{elamvazhuthi2023dynamical}. Furthermore, we relax the strong reachability assumption made in \cite{elamvazhuthi2023dynamical}, where it was required that every point in the support of the target measure be reachable from every point in the support of the initial measure. Beyond these extensions, the arguments developed here significantly simplify the methodology of \cite{elamvazhuthi2023dynamical}. In particular, we establish a direct equivalence between the Kantorovich problem and the fluid-dynamical Benamou-Brenier formulation. See Theorem \ref{thm:equivkbb}. To the best of the author's knowledge, such an equivalence has not been previously established in the nonlinear setting considered in works such as \cite{agrachev2009optimal,figalli2010mass,hindawi2011mass,chen2016optimal}. A key novelty enabling this is the use of point-to-point controllability to explicitly construct dynamically transporting measures between arbitrary probability measures. See Proposition \ref{prop:marg2yng}.

In addition, under structural assumptions such as the absence of singular minimizing curves (for instance, in 2-generating driftless systems or controllable LTI systems), we leverage the existence of Monge solutions \cite{agrachev2009optimal,figalli2010mass,hindawi2011mass} to show that the Benamou-Brenier problem can be convexified by passing to momentum and measure coordinates. This reformulation renders the problem convex. Notably, a proof  this kind of convexification was previously only known for linear systems with minimum-energy costs \cite{chen2016optimal}. See Theorem \ref{thm:conbb}.

Finally, we address the question of uniqueness of solutions to the continuity equation under the constructed control laws. Assuming again the absence of singular minimizing curves, we show that the feedback controls obtained are regular enough to ensure uniqueness of weak solutions. This is a significant contribution, as it has remained an open issue in the literature on optimal transport with control costs. A key challenge is that optimal controls are typically not Lipschitz continuous in space---a condition that usually ensures uniqueness. While one might hope to apply general results such as those in \cite{ambrosio2014continuity}, optimal controls rarely satisfy the necessary regularity conditions. In particular, even if the controls have bounded variation (a property weaker than differentiability in the sense of Sobolev), it is still difficult to control the divergence of the associated vector fields, which makes it not possible to apply existing results on well posedness of the continuity equation. For example, in \cite{duprez2019approximate} it is shown that under such regularity conditions exact controllability is not possible. To overcome these difficulties, we adopt a different strategy by we exploiting the structure of the associated Hamilton-Jacobi-Bellman (HJB) equation, following techniques inspired by \cite{cannarsa2004semiconcave, cannarsa2008semiconcavity, rifford2009stabilization}. See Theorems \ref{thm:uniquesol1} and \ref{thm:uniquesol2}.

\section{Notation and Assumptions}
In this section, we define some notation that will be used for the analysis later. Firstly, to address the problem of well-posedness of the optimal transport problem, instead of densities $\rho(t,x)$, we will work with measures that are not necessarily absolutely continuous with respect to the Lebesgue measure. 

Let $U \subseteq \R^n$ be the control set. For notational convenience we will define the control-dependent vector-field $f: \mathbb{R}^d \times U \rightarrow \Rd$, given by,
\begin{equation}
    f(x,u) = f_0(x)+ \sum_{i=1}^n u_i f_i(x)
\end{equation}
for each $(x,u) := (x, [u_1...u_n]^T) \in \Rd \times U $. 

We will need an appropriate notion of the solution of the continuity equation \eqref{eq:FP}. Towards this end, given a topological space $X$, we will denote by $\mathcal{P}(X)$ the set of Borel probability measures on $X$. We will often use the {\it narrow topology} on the set $\mathcal{P}(X)$, which is the coarsest topology such that the maps $f \mapsto \int_{X}f(x) d\mu(x)$ are continuous for all $f \in C_b(X)$, where $C_b(X)$ is the set of bounded continuous functions on $X$. For $p \geq 1$. We will need $\mathcal{P}_p(\Rd)$, the set of probability measures with finite $p^{th}$ moment defined by

\[\mathcal{P}_p(\Rd) : = \{ \mu \in \mathcal{P}(\Rd); \int_{\Rd}|x|^p d\mu(x)\}\]

We will say that a sequence $\mu_n \in \mathcal{P}(X)$ is narrowly converging to $\mu \in \mathcal{P}(X)$ if $\int_{X}f(x) d\mu_n(x) \rightarrow \int_{X}f(x) d\mu(x)$ for all $f \in C_b(X)$.
The solutions of the PDE \eqref{eq:FP} will be considered in the following sense. Let $I = [0,T]$ for $T>0$. We will say that a narrowly continuous curve $\mu :I \rightarrow \mathcal{P}(\Rd)$, i.e., continuous from $I$ to the narrow topology on $\mathcal{P}(\Rd)$, solves the PDE 
\begin{equation}\label{eq:FPM}
\partial_t \mu+  \nabla \cdot \left(f(\cdot, u)) \mu \right) =0, ~ \text{in}~ \Rd \times I
\end{equation}
in a weak sense, with initial and terminal conditions, $\mu_0 \in \mathcal{P}(\Rd)$ and $\mu_T \in \mathcal{P}(\Rd)$ respectively, if the following holds,
    \begin{eqnarray}
    \int_{I }\int_{\Rd} [ \partial_t \phi + \partial_x\phi \cdot f(x,u(t,x))]  d\mu_t(x)dt \nonumber \\= \int_{\Rd}\phi(T,x)d\mu_T(x) - \int_{\Rd}\phi(0,x)d\mu_0(x) 
    \label{eq:weTaq}
    \end{eqnarray}
    for all compactly supported once differentiable functions $\phi \in C^{1}_c (I \times \Rd)$,  where, $\partial_x \phi $ denotes the differential of $\phi $. Therefore, the optimization problem of interest is the following,
\begin{align}\label{eq:main_optimbb}
C_{BB}(\mu_0,\mu_T):=\inf_{u_i,\mu} \int_I\int_{M}   L(t,x,u(t,x)) d\mu_t(x)dt
 \nonumber \\
\text{subject to the constraint \eqref{eq:weTaq}}
\end{align}

   We will also need some assumptions on the running cost function $L(t,x,u)$. Following are the well-known as the {\it Tonelli conditions} on the Lagrangian $L(t,x,u)$ that are commonly used to establish existence of solutions of optimal control problems \cite{clarke2013functional}. Similarly, they will also play key role in establishing existence of solutions of the optimal transport problem.
   
\begin{assumption}
The running cost function $L:[0,T] \times \Rd \times U \rightarrow \mathbb{R}$ satisfies the following conditions,
\begin{enumerate}

 \item \textbf{(Continuity)} The running cost $L(t,x,u)$ is continuous.
 \item  \textbf{(Coercivity)} The set $U$ is closed and convex, and either one of the conditions hold.
 \label{asmp:costco}
  \begin{enumerate}
  \item 
 There exists $p > 1$, $\alpha> 0$ and $\beta \in \R$ such that $L(t,x,u) \geq \alpha |u|^p +\beta$ for all $ (t,x,u) \in [0,T] \times \Rd \times U$.
 \item The set $U$ is compact.
 \end{enumerate}
 \item  \textbf{(Convexity)}  The function $L(t,x, \cdot)$ is convex for each $(t,x) \in [0,T] \times  \Rd$.
\end{enumerate}
\end{assumption}

\begin{assumption}
The drift and control vector fields satisfy the following assumption.
\begin{enumerate}
\label{asmp:sublin}
\item There vector fields $f_i$  are $C^1(\Rd;\Rd)$ for each $i =0,...,n$. 
\item The vector fields $f_i$ have sublinear growth for each $i =0,...,n$. That is, there exists $M>0$ such that 
\[|f_i(x)| \leq M(|x|+1)\]
for all $x \in \Rd$ for $i = 0,..,n$.
\end{enumerate}
\end{assumption}

    In order to address the issue of existence of solutions to \eqref{eq:weTaq}-\eqref{eq:main_optimbb}, we will first consider a {\it relaxed} version of the problem, where instead of looking for control laws that assign to each $(t,x)$ a fixed control in $U$, we search instead for {\it Young measure} \cite{young2000lectures,florescu2012young} that assigns to each $t \in [0,T]$ a Borel probability measure on $\Rd \times U$. Towards this end, let $X$ be a topological space and $\mathcal{B}(X)$, the Borel sigma algebra. 
    We will need the projections maps $\pi^t:I \times \Rd \times U \rightarrow I$, $\pi^x:I \times \Rd \times U \rightarrow \Rd$ and $\pi^u:I \times \Rd \times U \rightarrow U$ defined by 
    
    \begin{align}
    \pi^t(t,x,u) = t, ~~
    \pi^x(t,x,u) = x, ~~
     \pi^u(t,x,u) = u \nonumber
    \end{align}
    for all $(t,x,u) \in I \times \Rd \times U$.
    We denote by $\mathcal{Y}(I;X) = \{ K \in \mathcal{P}(I \times X); \pi^t_{\#}K = {\rm leb}\}$,  where ${\rm leb}$ is the Lebesgue measure on $I$.  Note that given any $K \in \mathcal{P}(I \times X)$ with marginal $\pi^t_{\#}K = {\rm leb}$ , there exists a corresponding {\it disintegration} $K_t$ such that 
    \[ \int_{I \times X}f(t,x)dK(t,x) = \int_{I}\int_{X} f(t,x) dK_t(x)dt \]
   for all functions $f \in C_b(I \times X)$. Therefore, when me make the abuse of notation $K_t$, we mean the disintegration of $K$ with respect to the time variable evaluated at $t$. In this way, we can identify the subset of $\mathcal{Y}(I;X)$ with the set of measurable maps
    $I \ni t\mapsto K_t(\cdot) \in \mathcal{P}( X)$. By measurable, we mean that the function $t \mapsto K_t(A)$ is measurable for each Borel set $A \in \mathcal{B}(X)$. We will often use the narrow topology on $\mathcal{Y}(I;X)$, which is the smallest topology such that the functional $f \mapsto \int_0^T \int_{X} f(x)dK(t,x)$ is continuous for all $f \in C_b(X)$. We will say that a sequence $K^n \in \mathcal{Y}(I;X)$ is narrowly converging to a limit $K \in \mathcal{Y}(I;X)$ if $\int_{I \times X} f(t,x)dK^n(t,x) =\int_{I} \int_X f(t,x)dK^n_t(x)dt \rightarrow  \int_{I \times X} f(x)dK(t,x) = \int_I \int_{X} f(t,x)dK_t(x)$ for all $f \in C_b( X)$. 
  
 We will first establish that, for given $\mu_0 \in \mathcal{P}(\Rd)$ and $\mu_T \in \mathcal{P}(\Rd)$, there exists a $K\in \mathcal{Y}(I \times \Rd \times U)$  such that
    \begin{eqnarray}
    \label{eq:wkwkeq}
    \int_{I \times \Rd \times U}  [\partial_t \phi  +\partial_x \phi \cdot f(x,u) ]dK(t,x ,u) \nonumber \\= \int_{\Rd}\phi (T,x)d\mu_T(x) - \int_{\Rd}\phi (0,x)d\mu_0(x) 
    \end{eqnarray}
    for all functions $\phi  \in C^{1}_c (I \times \Rd)$.
    Given $\mu_0,\mu_T \in \mathcal{P}(\Rd)$, $\mathcal{Y}_{\mu_0}^{\mu_T}(I \times \Rd \times U)$ we will denote the set of $K \in \mathcal{Y}(I \times \Rd \times U)$ such that \eqref{eq:wkwkeq} holds.

Let $\Gamma = C([0,T];\Rd)$. We define the control set $\mathcal{U}$ and the set of admissible trajectories $\Omega$ in the following way.  
\[\mathcal{U} = \{ u \in L^p(0,T;\R^m); u(t) \in U,~\text{for a.e}~t \in [0,T] \}.\] 
\[\Omega := \{\omega \in  \Gamma, u\in \mathcal{U}; (\omega,u) ~\text{satisfy}~ \eqref{eq:ctrsyspre}   \}
.\]
Given the set of admissible trajectory, control pairs $\Omega$, suppose $S:\mathbb{R}^d \times \mathbb{R}^d \rightarrow \Omega$ is a map. We define $S^\omega_t(x,y)$ and $S^u_t(x,y)$ as the coordinate evaluations of the map $S$ at time $t$. Particularly, if $S(x,y) = (\ell,\alpha)$, then $S^\omega_t(x,y) = \ell(t)$ and $S^u_t(x,y) = \alpha(t)$ for almost every $t \in (0,T)$. 

We define the set of optimal trajectories $G \subset \Omega \times \mathcal{U}$, defined by 

\[G = \{ (\omega,u) \in \omega \times \mathcal{U}; J(\omega,u) = c(\omega(0),\omega(1) \}\]
    
Given these definitions, we will first consider the following \textbf{Relaxed Benamou-Brenier} problem:
 \begin{equation}\label{eq:main_optimrbb}
\begin{split}
\inf_{K \in \mathcal{Y}_{\mu_0}^{\mu_T}(I \times \Rd \times U)} \int_{I \times \Rd \times U}   L(t,x,u) dK(t,x,u)dt 
\end{split}
\end{equation}

It can be seen that the above problem is a relaxation of the Benamou-Brenier problem \eqref{eq:weTaq}-\eqref{eq:main_optimbb} by noting that if $\mu_t$ is a weak solution of the PDE for the control $u:[0,T] \times M \rightarrow U$, then the Young measure $K_t$ given by $K_t (A) = \int_{A} \delta_{u(t,x)}(A)d\mu_t(x)$ for all $A \in \mathcal{B}(\Rd \times U)$, satisfies  \eqref{eq:wkwkeq}.

\section{Analysis}

In this section we present the main analysis of the paper, starting with the first contribution of the paper, which is, establishing the equivalence between the Kantororovich problem and the dynamical formulation of the problem.

\subsection{Equivalence between Kantorovich and Benamou-Brenier Formulation}

\begin{proposition}
 Given Assumption \ref{asmp:costco} (Tonelli running cost) and \ref{asmp:sublin} (vector-fields with sub-linear growth), the cost $c(x,y)$ is lower semicontinuous. Hence, a solution of the the Kantorovich problem \eqref{eq:kant} exists if the feasible set it non-empty.
\end{proposition}

\begin{proof}
Let $r \in \mathbb{R}$ and $(x_n,y_n)_{n=1}^{\infty}$ be a sequence in $\Rd \times \Rd$ converging to  $(x,y) \in \Rd \times \Rd$  such that $(c(x_n,y_n))_{n=1}^{\infty} $ is a Cauchy sequence  converging to some $r \in \mathbb{R}$ that is finite. There exist optimal solutions $(\omega_n,u_n)$ corresponding to the costs $c(x_n,y_n)$ due standard existence results for optimal control problems \cite[23.11]{clarke2013functional}. We know that 
$\int_0^TL(t,\omega_n(t),u_n(t))dt$ is uniformly bounded. This implies that $u$ is uniformly bounded in $L^{p}(0,T;\R^m)$. This implies that we can take a subsequence  $(\omega_{n_m},u_{n_m})_{m=1}^{\infty}$ such that $\omega_{n_m}$ is strongly converging to some $\omega$ in $C([0,T];\Rd)$ and $u_{n_m}$ is weakly converging to $u$ in $\mathcal{U}$ as in proof of \cite[Theorem 23.11]{clarke2013functional}. We know that the functional $J$ is lower semicontinuous, which gives us
\[r = \liminf_{m \rightarrow \infty} J(\omega_{n_m},u_{n_m}) \geq J(\omega,u) \geq  c(x,y) \]
This implies that the preimage of every set of the form $ [-\infty,a)$ under the function $c$ is closed for all $a \in \R$. Therefore, $c$ is lower semicontinuous. The existence of solution now follows from classical results in optimal transport \cite[Theorem 1.7]{santambrogio2015optimal}.
\end{proof}

Using the same idea as in the previous Proposition we also have the following result.

\begin{lemma}
Given Assumption \ref{asmp:costco} (Tonelli running cost) and \ref{asmp:costco} (vector-fields with sub-linear growth), the set of optimal trajectory control pairs $G$ is closed in $\Omega$.
\end{lemma}

In the following proposition we establish that the marginal of a Young measure can be associated with a curve on the set of probability measures.

\begin{proposition}
\label{prop:cntcur}
Let $\eta \in \mathcal{Y}(I \times \Rd \times U)$ be a controlled transport measure. There exists a continuous family $\mu: I \rightarrow \mathcal{P}(\Rd)$, such that $\mu_t = \pi^x_{ \#} \eta_t$  for all $t\in I$.
\begin{proof}
Define the map $\tilde{f}:I \times \Rd \times U \rightarrow I \times \R^{2d}$ by
\[\tilde{f}(t,x,u) = (t,x,f(x,u))\]
for all $(t,x,u) \in I \times \Rd \times U$.
Let $\psi = \tilde{f}_{\#}\eta$. Then $\psi \in \mathcal{Y}(I; \R^{2d})$ satisfies,
\begin{equation}
\int_{I \times \Rd \times U} [\partial_tg + \nabla g \cdot v]d\Psi(t,x,v) = 0
\end{equation}
for all compactly supported smooth functions $g \in C^{1}_c(I \times \Rd)$. The result then follows from \cite[Lemma 5]{bernard2008young}, by noting that $\pi^x_{\#}\eta = \pi^x_{\#}\psi  $. Particularly, one has that there exists a continuous family of $\mu_t: I \rightarrow \mathcal{P}(\Rd)$, such that $\mu_t = \pi^x_{\#} \Psi_t$  for all $t\in I$. Hence, it also follows that $\mu_t = \pi^x_{\#} \eta_t$  for all $t\in I$.
\end{proof}
\end{proposition}

In the following result, we establish how one can construct a Young measure with given marginals, using feasible trajectories connecting states on $\Rd$.

\begin{proposition}
\textbf{(Trajectories on $\Rd$ to Trajectories on $\mathcal{P}(\Rd)$)}
Let $\gamma \in \Pi$. Suppose $S:\mathbb{R}^d \times \Rd \rightarrow \Omega $ is a measurable map defined $\gamma$ almost everywhere on $\Rd \times \Rd$ such that 
\[S^{\omega}_0(x,y) = x, ~~ S^{\omega}_T(x,y) = y\]
for $\gamma$ almost every $(x,y)\in \Rd \times \Rd$. 
Then there exists $\eta \in \mathcal{Y}^{\mu_T}_{\mu_0}(I \times \Rd \times U)$.
\label{prop:marg2yng}
\end{proposition}
\begin{proof}
We define the measure $\eta \in \mathcal{P}(I \times R^d \times U) $ by 
\begin{equation}
\int_{I \times \Rd} f(t,x,u)d\eta(t,x,u)  =  \int_{\Rd \times \Rd}\int_{I}f(t,S^\omega_t(x,y),S^u_t(x,y))dtd\gamma(x,y).
\end{equation}
for all $f \in C_b(I \times \mathbb{R}^d \times U)$, where $S^{\omega}_t$ and $S^{u}_t$ denote the first and second argument of $S(x,y)$ evaluated at $t$, respectively. Note that $\eta = F_{\#} \gamma$ is the pushforward of the measure $\gamma$ through the map $\Rd \times \Rd \ni  (x,y) \mapsto F(x,y) := (t,S^\gamma(x,y),S^u_t(x,y)) \in I \times \Rd \times U$.

We note that $\pi^t_{\#} \eta = \lambda $, the Lebesgue measure. To see this, we compute

\begin{align}
\int_{I} f(t) d\pi^t_{\#} \eta(t) &= \int_{I \times \mathbb{R}^{2d}}f(\pi^t_{\#}(t,x,v))d\eta(t,x,v) \\
&=\int_{I \times \mathbb{R}^{2d}}f(\pi^t_{\#}(t,S^{\omega}_t(x,y),S^{u}_t(x,y)))dtd\gamma(x,y)  \\
&=\int_{I }f(t)dt 
\end{align}
This implies that there exists a disintegration $I  \mapsto \mathcal{P}(\mathbb{R}^{2} \times U)$, given by $ t\mapsto \eta_t$ such that $\eta_t \in \mathcal{P}(\mathbb{R}^{d} \times U)$ for $\lambda$ almost every $ t\in [0,T]$. 

Let $g \in C^1_c(I \times \Rd)$. We know that 
\begin{align}
\frac{d}{dt}(g(t,S^{\omega}_t(x,y)) = \partial_t g(t,S^{\omega}_t(x,y)) + \partial_x g(t,S^{\omega}_t(x,y)) \cdot f(S^{\omega}_t(x,y),S^{u}_t(x,y))
\end{align}
for almost every $t \in [0,T]$,  $\gamma$ almost every $(x,y) \in \mathbb{R}^{2d}$. Therefore, 
\begin{align}
& g(T,S^{\omega}_T(x,y)) - g(0,S^\omega_0(x,y))   \\
& = \int_I \partial_t g(t,S^{\omega}_t(x,y)) + \partial_x g(t,S^{\omega}_t(x,y)) \cdot f \big ( S^{\omega}_t(x,y),S^{u}_t(x,y) \big )dt
\end{align}

for $\gamma$ almost every $(x,y) \in \mathbb{R}^{2d}$.

Integrating both sides with respect to the measure $\gamma$, we get

\begin{align}
\int_{\Rd}g(T,x)d\mu_T - \int_{\Rd}g(0,x)d\mu_0  =\int_{\mathbb{R}^{2d}} \int_{I}[\partial_t g(t,x) + \partial_x g(t,x) \cdot f(x,u)]d\eta(t,x,u) 
\end{align}
\end{proof}

Using the above proposition we can establish that non-emptiness of the feasible set of the Kantorovich problem, implies non-emptiness of the feasible set of the Benamou-Brenier problem.

\begin{proposition}
\label{prop:kan2bb}
Given Assumption \ref{asmp:costco} (Tonelli running cost) and \ref{asmp:costco} (vector-fields with sub-linear growth).  Let $\gamma \in \Pi$ be such that $\int_{\Rd \times \Rd} c(x,y)d\gamma(x,y) < \infty$. Then there exists a controlled transport measure $\eta \in \mathcal{Y}^{\mu_0}_{\mu_1}(I \times \Rd \times U)$ such that 
\begin{equation}
\int_{I \times \Rd \times U} L(t,x,u)d\eta(t,x,u) = \int_{\Rd \times \Rd} c(x,y)d\gamma(x,y)
\end{equation}
\end{proposition}
\begin{proof}
Let $\Gamma \subseteq \Rd \times \Rd$ be defined by
\[\Gamma : = \{ (x,y) \in \Rd \times \Rd; c(x,y) <\infty\}.\] 
Let $P^G$ denote the set of all subsets of $G$. We define the multivalued mapping $\Psi : \Gamma \rightarrow P^G$ 
\begin{equation}
\Psi(x,y) = \lbrace (\omega,u) \in G; \omega(0)=x,~\omega(T) = y\rbrace 
\end{equation}
for all $(x,y) \in \Gamma$. 
By \cite[Theorem 6.9.13]{bogachev2007measure} there exists a map $S: \Rd \times \Rd \rightarrow \Omega \times \mathcal{U}$ that is measurable (with respect to the Lebesgue $\sigma$-algebra on $\Omega$ and the Borel $\sigma$-algebra on $\Omega$) such that 
\[S(x,y) \in \Psi(x,y)\] 
for all $x, y \in \Omega$. We can extend $S$ to $\Rd \times \Rd$ such that it is defined $\gamma$ everywhere. With an abuse of notation, we denote this extension by $S:\Rd \times \Rd \rightarrow \Omega \times \mathcal{U}$. The existence of $\eta \in \mathcal{Y}^{\mu_T}_{\mu_0}(I \times \Rd \times U)$ follows from Proposition \ref{prop:marg2yng}. Next, we compute $\int_{I \times \Rd \times U} L(t,x,u)d\eta(t,x,u)$,
\begin{align*}
\int_{I \times \Rd \times U} L(t,x,u)d\eta(t,x,u) &=   \int_{\Rd \times \Rd} \int_{I}L(t,S^\omega_t(x,y),S^u_t(x,y))dtd\gamma(x,y) \nonumber \\
& \int_{\Rd \times \Rd} c(S^\omega_0(x,y),S^\omega_T(x,y))d\gamma(x,y) \nonumber \\
&  \int_{\Rd \times \Rd} c(x,y)d\gamma(x,y)
\end{align*}
\end{proof}

Once one has feasibility, a standard approach to establish existence of solutions is to establish compactness of any minimizing sequence of solutions. Using the growth conditions of the vector-fields, the following Lemma provides some estimates that will be used to derive the compactness.

\begin{lemma}
\label{lem:comp}
Given Assumption \ref{asmp:costco} (Tonelli running cost) and \ref{asmp:sublin} (vector-fields with sub-linear growth). Let $\mu_0 \in \mathcal{P}_p(\Rd)$. Suppose $\eta \in \mathcal{Y}^{\mu_0}_{\mu_1}(I \times \Rd \times U)$ such that \\ $\int_{I \times \Rd \times U}L(t,x,u)d\eta(t,x,u) \leq c$ for some constant $c>0$. Then the curve $\mu : I \rightarrow \mathcal{P}(\Rd)$ given in Proposition \ref{prop:cntcur} satisfies the estimate,
\begin{equation}
\sup_{t \in I} \int_{\Rd}|x|^pd\mu_t(x) \leq (A +\int_{\Rd}|x|^pd\mu_0(x))e^{BT}.
\end{equation}
for some constants $A,B>0$ that depend only on $c$ and the constants $p, \alpha, \beta$ and $M$ in Assumptions \ref{asmp:costco} and Assumption \ref{asmp:sublin}.
Therefore, the set 
\begin{equation}
\{\eta \in \mathcal{Y}^{\mu_0}_{\mu_1}(I \times \Rd \times U) ; \int_{I \times \Rd \times U}L(t,x,u)d\eta(t,x,u) \leq c \}
\label{eq:cmpmeas}
\end{equation}
is compact in $\mathcal{P}(I \times \Rd \times U)$.
\end{lemma}
\begin{proof}
First we present a formal argument. 
Let $|x|^p$ be a test function. That is, we substitute of $g(t,x) = |x|^p$ in the definition \eqref{eq:wkwkeq}. We get the bound
\begin{align}
\int_{\Rd}|x|^pd\mu_t - \int_{\Rd}|x|^pd\mu_0  &=\int_{I \times \Rd \times U} [\partial_t |x|^{p} + \partial_x |x|^p \cdot f(x,u)]d\eta(t,x,u)  \nonumber \\
& = p\int_{I \times \Rd \times U} [ |x|^{p-1}\cdot f(x,u)]d\eta(t,x,u) \nonumber \\
& \leq \int_{I \times \Rd \times U} M[(|x|^{p-1}+|x|^p+|u||x|^{p-1})]d\eta(t,x,u) \nonumber \\
& \leq \int_{\Rd \times U} M\int_{I}[(|x|^{p-1}+|x|^p+\frac{|u|^p}{p}+\frac{(p-1)|x|^{p-1}}{p})]d\eta(t,x,u) \nonumber \\
\nonumber \\
& \leq A +B \int_{I \times \Rd }|x|^pd\mu_t(x)
\end{align}
using Young's inequality and using the fact that $|x|^p \geq |x|^{p-1}$ outside the unit ball, and
$A>0$ and $B>0$ only depend on $c$.
Then by Gronwall's equality we get,
\begin{align}
\int_{\Rd}|x|^pd\mu_t(x) \leq (A+ \int_{\Rd}|x|^pd\mu_0(x)) e^{Bt}
\label{eq:es1}
\end{align}
for all $t \in I$.
There are two issues with this argument. First, $|x|^p$ is not a compactly supported function. Second, it is not differentiable everywhere for all $p$.
To make this argument rigorous, we can approximate $|x|^p$ using a smooth function $g_{\varepsilon}$ such that $g_\varepsilon \uparrow |x|^p$ as in proof of Lemma 5.1 in \cite{fornasier2019mean}, and then recover the estimate \eqref{eq:es1} by taking the limit.

Next, we establish the compactness of the set
\eqref{eq:cmpmeas}
The estimate \eqref{eq:es1} implies that 
$\int_{I \times \Rd \times U}|x|^pd\eta(t,x,u) \leq c'$
where $c'>0$ depends only on $c$ and $\mu_0$. The definition of the set \eqref{eq:cmpmeas} and the estimate \eqref{eq:es1} implies that 
$\int_{I \times \Rd \times U}(|x|^p+\alpha |u|^p-\beta)d\eta(t,x,u) \leq c' +c$
where $\alpha>0$ and $\beta$ are the constants in the assumption of coercivity of the running cost function \ref{asmp:costco}.
The function $(t,x,u) \mapsto (|x|^p+\alpha |u|^p-\beta)$ has compact level sets. Therefore, the set \eqref{eq:cmpmeas} is tight by \cite[Corollary 3.61]{florescu2012young}
\end{proof}

The next theorem establishes a {\it purification result} that given a measure value control law for the relaxed problem, one can construct a deterministic control law that achieves a lower cost.

\begin{proposition}
\label{prop:pur}
Given Assumption \ref{asmp:costco} (Tonelli running cost) and \ref{asmp:costco} (vector-fields with sub-linear growth). Let $\eta \in \mathcal{Y}^{\mu_T}_{\mu_0}(I;\R^{d} \times U )$ be such that 
\[ \int_{I \times \Rd \times U }L(t,x,u)d\eta(t,x,u) < \infty \]
there exists a measurable control $u: [0,T] \times \mathbb{R}^d \rightarrow U$ such that 
\[\int_{I \times \Rd \times U} [\partial_t g + \nabla_x g \cdot f(x,u)] d\eta(t,x,u) = \int_{I \times \Rd } [\partial_t g + \nabla_x g \cdot f(x,u(t,x))]d\mu_t(x)    \]
for all $g \in C^{\infty}_c(I \times \Rd)$. Moreover, 
\begin{equation}
    \int_{I \times \Rd \times U }L(t,x,u)d\eta(t,x,u) \geq    \int_{I} \int_{\Rd \times U }L(t,x,u(t,x))d\mu_t(x)dt
\end{equation}
\end{proposition}
\begin{proof}
Since $\pi^x_{\#} \eta_t = \mu_t$, there exists a disintegration of measure $ I \times \Rd \ni  (t,x) \mapsto \eta_{t,x} \in \mathcal{P}(U)$ such that

\begin{align}
\int g(t,x,v) d \eta_t(x,u) =  \int g(t,x,u) d \eta_{t,x}(v) d\mu_t(x)
 \end{align}
for all compactly supported functions $g \in C^{\infty}(I \times \Rd \times U)$.
This implies 
\[\int_{I \times \Rd \times U} [\partial_t g + \nabla_x g \cdot f(x,u)] d\eta(t,x,u) = \int_{I \times \Rd \times U} [\partial_t g + \nabla_x g \cdot f(x,u)] d\eta_{t,x}(u)d\mu_t(x)dt \]
Since $f(x,u)$ is affine with respect to $u$, 
we get

\begin{align}\int_{I \times \Rd \times U} [\partial_t g + \nabla_x g \cdot f(x,u)] d\eta(t,x,u) \nonumber \\= \int_{I \times \Rd} [\partial_t g + \nabla_x g \cdot f(x, \int_{U} u d\eta_{t,x}(u)]d\mu_t(x)dt  \nonumber \\
\end{align}

Defining $u(t,x) = \int_{U}u d\eta_{t,x}(u)$ for $\mu_t$ almost every $x \in \Rd$ and Lebesgue almost every $ t \in I$, we have the first part of our result.
Next, by assumption $u \mapsto L(t,x,u)$ is convex for every $(t,x) \in I \times \Rd$. Then by Jensen's inequality we have that,

\begin{align} \int_{I \times \Rd \times U }L(t,x,u)d\eta(t,x,u) &=    \int_{I \times \Rd  }\int_U L(t,x,u)d\eta_{t,x}(u)d\mu_t(x) dt
\nonumber \\
&\geq \int_{I \times \Rd  } L(t,x,\int_U ud\eta_{t,x}(u))d\mu_t(x) dt \nonumber \\
& = \int_{I \times \Rd } L(t,x,u(t,x))d\mu_t(x) dt
\end{align}
This concludes the proof.
\end{proof}

Given the above result we can conclude the following corollary of constructing a feedback control that transports a given measure to a final one, based on controllability of the system \eqref{eq:ctrsyspre}.

\begin{corollary}
\textbf{(Point-to-point controllability to Measure-to-measure controllability)}
Let $\gamma \in \Pi$. Suppose $S:\mathbb{R}^d \times \Rd \rightarrow \Omega $ is a measurable map defined $\gamma$ almost everywhere on $\Rd \times \Rd$ such that 
\[S^{\omega}_0(x,y) = x, ~~ S^{\omega}_T(x,y) = y\]
for $\gamma$ almost every $(x,y)\in \Rd \times \Rd$ and 
\[\int_0^T\int_{\Rd \times \Rd} L(t,S^{\omega}_0(x,y),S^{u}_0(x,y))dtd\gamma(x,y)<\infty.\] Then there exists a pair $(\mu,u)$ that solves the continuity equation \eqref{eq:weTaq}, such that $\mu_0 = \pi^1_{\#} \gamma$ and $\mu_T = \pi^2_{\#} \gamma$, and $u(t,\cdot) \in L^p(\mu_t)$ for almost every $t \in [0,T]$.
\end{corollary}
\begin{proof}
The proof follows as in proof of \ref{prop:kan2bb} and then applying Proposition \ref{prop:pur}.
\end{proof}

Next, we are able to use the compactness to establish existence of minimizers given that one assumes the feasible set is non-empty.
\begin{theorem}
\label{thm:existbb}
\textbf{(Feasibility implies existence of minimizer)}
Given Assumption \ref{asmp:costco} (Tonelli running cost) and \ref{asmp:costco} (vector-fields with sub-linear growth). Let $\mu_0 \in \mathcal{P}_p(\Rd)$ and $\mu_T \in \mathcal{P}_p(\Rd)$. Suppose the relaxed Benamou-Brenier problem \eqref{eq:main_optimrbb} is feasible. Then there exists a pair $(\mu,u)$ that solves the Benamou-Brenier formulation of optimal transport \eqref{eq:main_optimbb}. 
\end{theorem}
\begin{proof}
    Since the optimization problem \eqref{eq:main_optimrbb} is feasible and the set of decision variables $\mathcal{Y}^{\mu_T}_{\mu_0}(I :\Rd \times U)$ is compact by Lemma \ref{lem:comp}. The running cost $L$ is continuous and bounded from below. Hence, the map $\eta \mapsto \int_{I \times \Rd \times U}L(t,x,u)\eta(t,x,u)$ is lower semi-continuous by the Portmanteau theorem \cite[Lemma 5.1.7]{ambrosio2008gradient}. Therefore, \eqref{eq:main_optimrbb} has a minimizer. According to Proposition \ref{prop:pur} there exists a pair $(\mu, u)$ that achieves a lower cost that
\begin{equation}
    \int_{I \times \Rd \times U }L(t,x,u)d\eta(t,x,u) \geq    \int_{I} \int_{\Rd  }L(t,x,u(t,x))d\mu_t(x)dt
\end{equation}

\end{proof}

The next result establishes an important equivalence between the Kantorovich problem and the Benamou-Brenier problem.
\begin{theorem}
\label{thm:equivkbb}
\textbf{(Equivalence of Kantorovich and Benamou-Brenier formulation)}
Given Assumption \ref{asmp:costco} (Tonelli running cost) and \ref{asmp:costco} (vector-fields with sub-linear growth). Let $\mu_0, \mu_T \in \mathcal{P}_1(\Rd)$. The Kantorovich problem \eqref{eq:kant} and the Benamou-Brenier problem \eqref{eq:main_optimbb} are equivalent,
\[C_{\rm kan}(\mu_0,\mu_T) = C_{BB}(\mu_0,\mu_T)\]
\end{theorem}
\begin{proof}
Let $(\mu,u)$ be an optimal pair that minimizes \eqref{eq:main_optimbb}. From the coercivity assumption on the running cost and we can infer that 

\[ \int_{I} \int_{\Rd} \frac{|f(x,u(t,x))|}{1+|x|}d\mu_t(x)dt \leq \int_{\Rd} M(|u(t,x)|)d\mu_t(x)dt < \infty \]
for some $M>0$, due to the assumptions on the running cost $L$.
Since, $\mu$ solves the continuity equation, from the superposition principle \cite[Theorem 3.4]{ambrosio2014continuity}, we can infer that there exists a probability measure $\eta \in \mathcal{P}(\Rd \times \Gamma)$ such that $(e_t)_{\#} \eta = \mu_t$ for all $t \in [0,T]$, 
and 
\[\dot{\omega}= f(\omega(t),u(t,\omega(t))).\]Consider the map $ (x,\omega) \mapsto E(x,\omega):= (e_0(\omega),e_T(\omega))$ and the measure $E_{\#} \eta \in \mathcal{P}(\Rd \times \Rd)$. Computing
\begin{align}
 \int_{\Rd \times \Rd}c(x,y)dE_{\#} \eta(x,y)  & =\int_{\Rd \times \Gamma}c(\omega(0),\omega(T))d \eta(x,\omega) \nonumber \\ 
& \leq \int_{I}\int_{\Rd \times \Gamma} L(t,\omega(t),u(t,\omega(t)))d \eta(x,\omega) dt\nonumber \\
& = \int_I \int_{\Rd } L(t,x,u(t,x)) d\mu_t(x)dt \nonumber \\
\end{align}
This in combination with Proposition \eqref{prop:kan2bb} imply that both the optimization problems achieve the same minimum.
\end{proof}

In the following proposition we note that the optimal solutions are concentrated on optimal curves as in solutions of Benamou-Brenier for Euclidean costs. Though the proof is similar to that of the previous result, we state it here due to its importance. 
\begin{proposition}
\textbf{(Concentration of transport on optimal trajectories)}
Given Assumption \ref{asmp:costco} (Tonelli running cost) and \ref{asmp:costco} (vector-fields with sub-linear growth). Let $\mu_0, \mu_T \in \mathcal{P}_1(\Rd)$. Suppose there exists an optimal pair $(\mu,u)$ solving the Benamou-Brenier problem \eqref{eq:main_optimbb}. Then there exists a probability measure $\eta \in \mathcal{P}(\Rd \times \Gamma)$ such that $(e_t)_{\#} \eta = \mu_t$ for all $t \in [0,T]$, 
and 
\[\dot{\omega}= f(\omega(t),u(t,\omega(t)))\]
Moreover, the probability measure $\eta$ satisfies, ${\rm supp} ~ \eta \subseteq \Rd \times \pi^\omega(G):= \{\omega; {\rm s.t.}~ (\omega,u) \in G \}$.
\end{proposition}
\begin{proof}
The existence of $\eta$ follows using an application of the superposition principle as in Theorem \ref{thm:equivkbb}. For the concentration of the curves, we can once again define the map $ (x,\omega) \mapsto E(x,\omega):= (e_0(\omega),e_T(\omega))$ and the measure $E_{\#} \eta \in \mathcal{P}(\Rd \times \Rd)$. We  know from Theorem \ref{thm:equivkbb} that
\begin{align}
 \int_{\Rd \times \Rd}c(x,y)dE_{\#} \eta(x,y)  & =\int_{\Rd \times \Gamma}c(\omega(0),\omega(T))d \eta(x,\omega) \nonumber \\ 
& = \int_{I}\int_{\Rd \times \Gamma} L(t,\omega(t),u(t,\omega(t)))d \eta(x,\omega) dt \nonumber
\end{align}
If $\omega$ is not optimal for $\eta$ almost every $x,\omega$, then the last equality would lead to a contradiction.
\end{proof}

The curve $(e_t)_{\#} \eta$ is known as the displacement interpolation between $\mu_0$ and $\mu_T$. When there exists a solution to the Monge problem the measure $E_{\#} \eta = (T, Id)_{\#} \mu_0 $, where $T_{\#}\mu_0 = \mu_T $ is the deterministic map that transports $\mu_0$ to $\mu_T$ and solves the Monge problem \ref{eq:mong}.

Next, we consider some special cases. 

Let $\mathcal{V} = \{ g_1,...,g_m\}$. Let $[f,g]$ denote the Lie bracket operation between two vector fields $f: \mathbb{R}^d \rightarrow \mathbb{R}^d$ and $g: \mathbb{R}^d \rightarrow \mathbb{R}^d$, given by
where $\partial_{i}$ denotes partial derivative with respect to coordinate $i$.
 \begin{equation}
 [f,g]_i = \sum_{j=1}^d f^j \partial_{x_j} g^i - g^j \partial_{x_j} f^i.
 \end{equation}
 
 We define $\mathcal{V}^0 =\mathcal{V}$. For each $i \in \mathbb{Z}_+$, we define in an iterative manner the set of vector fields $\mathcal{V}^i = \lbrace [g, h]; ~g \in \mathcal{V}, ~h \in \mathcal{V}^{j-1}, ~j=1,...,i \rbrace$.  We will assume that the collection of vector fields $\mathcal{V}$ satisfies following condition the {\it Chow-Rashevsky} condition \cite{agrachev2019comprehensive}. We will say that $\cup_{i =0}^r \mathcal{V}^i$ has rank $d$ if ${\rm span} \{ g(x) \in \cup_{i =0}^r \mathcal{V}^i \} = \Rd$ for all $x \in \Rd$.
 
 \begin{assumption}
 \label{asmp:ctbdrif}
 \textbf{(Controllable driftless system/Hormander condition})
 Suppose $f_0 \equiv 0$ and $f_i \in C^{\infty}(\Rd ;\Rd)$ for each $i =1,...,n$. The Lie algebra of step $r$ generated by the vector fields $\mathcal{V}$, given by $\cup_{i =0}^r \mathcal{V}^i$, has rank $d$, for sufficiently large $r$.
 \end{assumption}

Consider the {\it sub-Riemannian metric} $d : \Rd \times \Rd \rightarrow \R$
 \begin{align}
d(x,y) := \inf_{\omega,u} \int_0^T \sqrt{\sum_{i=1}^m u^2_i(t)}dt \nonumber \\
 \text{subject to}~~\eqref{eq:ctrsyspre}.
 \end{align}

We note the following equivalence between the sub-Riemannian energy cost and the optimal control problem with kinetic energy running cost.

 \begin{lemma}(\cite{cannarsa2008semiconcavity})
 Suppose Assumption \ref{asmp:ctbdrif} holds. Then $d^2(x,y) = c(x,y)$ for all $(x,y) \in \Rd \times \Rd$, where the running cost is given by $L(t,x,u) = \sum_{i=1}^m|u_i|^2$ for all $(t,x,u) \in I \times \Rd \times U$.
 \end{lemma}

Given this equivalence we can state the following result on the existence of solutions for the Benamou-Brenier problem for the Sub-Riemannian cost.

\begin{corollary}
\textbf{(Sub-Riemannian energy)}
Given Assumption \ref{asmp:costco}. Suppose additionally that the system \eqref{eq:ctrsyspre} is a controllable driftless system according to Assumption \ref{asmp:ctbdrif}. Let $\mu_0, \mu_T \in \mathcal{P}_c(\Rd)$ be measures that are compactly supported.  Let $L(t,x,u) = \sum_{i=1}^m|u_i|^2$ for all $(t,x,u) \in I \times \Rd \times U$ with $U = \R^d$. Then there exists a pair $(\mu,u)$ that solves the Benamou-Brenier formulation of optimal transport \eqref{eq:main_optimbb}. 
\end{corollary}
\begin{proof}
It is known 
from to Assumption \ref{asmp:ctbdrif}, $d(x,y)$ is well defined for all $(x,y) \in \Rd \times \Rd$, and continuous. See \cite[Theorem 3.31]{agrachev2019comprehensive}. 

Therefore, $c(x,y) = d^2(x,y)$ is continuous.  This implies that the Kantorovich problem \eqref{eq:kant} has a feasible solution. That is, one can take the product measure $\gamma_{so} = \mu_0 \otimes \mu_T$ and we have that $\int_{\Rd \times \Rd} c(x,y) d\mu_0(x)d\mu_T(y) < \infty$. This implies that the relaxed Benamou-Brenier formulation \eqref{eq:main_optimrbb} is feasible due to feasibility of the Kantorovich problem due to Proposition \ref{prop:kan2bb}. Then the result follows from Theorem \ref{thm:existbb} and Theorem \ref{thm:equivkbb}.
\end{proof}

Next we consider the class of controllable linear systems.

 \begin{assumption}
 \label{asmp:ctbllin}
 \textbf{(Controllable linear time invariant system})
Suppose $f_0(x) =Ax$ and $f_i(x) =b_i$ for $A \in \Rd \times \Rd$, and $b_i \in \Rd$ such that $(A,B):=(A,[b_1,...,b_m])$ satisfies the Kalman rank condition, 
\begin{equation}
{\rm rank}~ [B~AB~....A^{d-1}B] = d
\end{equation}
 \end{assumption} 

Given this assumption, we can state the following result for the well known linear quadratic costs. 
\begin{corollary}
\textbf{(Linear quadratic cost)}
Let $\mu_0, \mu_T \in \mathcal{P}_2(\Rd)$. Suppose the control system \eqref{eq:ctrsyspre} is a controllable LTI  according to Assumption \eqref{asmp:ctbllin}. Let the running cost be given by $L(t,x,u) = \langle x,Qx  \rangle_2 +\langle u,Ru  \rangle_2$ for some postive semidefinite matrix $Q \in \R^{d  \times d}$ and positive definite matrix $R \in \R^{n \times n}$,  for all $(t,x,u) \in I \times \Rd \times U$ with $U = \R^d$.  Then there exists a pair $(\mu,u)$ that solves the Benamou-Brenier formulation of optimal transport \eqref{eq:main_optimbb}. 
\end{corollary}
\begin{proof}
According to \cite{hindawi2011mass}, there exists matrices $D,E,F \in \R^{d \times d}$ such that $D$ and $F$ are positive definite and $E$ is invertible such that 
\[c(x,y) = \langle x,Dx\rangle_2 - \langle x,Ey\rangle_2+\langle y,Fy \rangle_2\]
for all $(x,y) \in \Rd \times \Rd$. Consider the product measure once again $\gamma_{so} = \mu_0 \otimes \mu_T$. We can see that 
\[c(x,y) \leq C(|x|^2 + |y|^2) \]
for all $(x,y) \in \Rd \times \Rd$ for some $C>0$, that is independent of $x$ and $y$. Since, $\mu_0, \mu_T \in \mathcal{P}_2(\Rd)$, we can infer that $c$ is integrable with respect to $\gamma_{so}$ and \[\int_{\Rd \times \Rd}c(x,y)d \gamma_{so} <\infty .\] This implies that the relaxed Benamou-Brenier formulation \eqref{eq:main_optimrbb} is feasible due to feasibility of the Kantorovich problem due to Proposition \ref{prop:kan2bb}. Then the result again follows from Theorem \ref{thm:existbb} and \ref{thm:equivkbb}.
\end{proof}

\subsection{Convexification of Benamou-Brenier problem}

In general, the Benamou-Brenier problem is not convex in the parameters $(\mu,u)$. However, for certain choices of costs one can convexify the problem in the following way, as done in \cite{benamou2000computational} for the case $L(t,x,u) = |u|^2$ and $f(x,u) = u $ with $m= d$. In this section, we briefly look at when such a convexification is possible for control constrained transport problems.
Towards this end let $\mathcal{M}([0,T] \times \Rd ; \R^m)$ be the set of vector valued signed Borel measures on $\Rd$ with range in $\R^m$.

We define the functional $\mathcal{F}: \mathcal{M}([0,T] \times \Rd ; \R^m) \times C([0,T];\mathcal{P}_1(\Rd)) \rightarrow \mathbb{R} \cup \{\infty\}$ by

\begin{equation} 
\mathcal{F}(m,\mu) =
\begin{cases}
\int_0^T \int_{\Rd} s(x)d\mu_t(x) + \int_0^T \int_{\Rd} K(m(t,x),\mu_t(x))dxdt, \\
~~~~~~ \text{if} ~~dm(x,t)= m(x,t)dxdt, ~~d\mu_t(x)dt= \mu_t(x)dxdt \\
+\infty ; \text{otherwise}
\end{cases}
\end{equation}
where $s:\mathbb{R}^d \rightarrow \mathbb{R}$ is the state cost and $\psi:\mathbb{R}^m \rightarrow \R$ and $K $ is the control cost defined by 
\begin{equation}
K(a,b) = 
\begin{cases}
b\psi(\frac{a}{b}) ;~~ \text{if}~ b>0 \\
0 ; ~~\text{if}~ a=0, ~ b =0 \\
+\infty ; ~~\text{if}~ a \neq 0, ~b = 0, ~\text{or}~ b<0
\end{cases}
\end{equation}
If $\psi$ is convex, then it can be verified that $K$ is convex. For instance, suppose $b_1, b_2>0$. Then 
\begin{align*}
K(\lambda_1 a_1 +\lambda_2 a_a, \lambda_1 b_1 +\lambda_1 b_2) = (\lambda b_1 + (1-\lambda )b_2) \psi (\frac{\lambda a_1 + (1-\lambda )a_2}{\lambda b_1 + (1-\lambda )b_2})  \\
 = (\lambda b_1 + (1-\lambda )b_2) \psi ( \tilde{\lambda} \frac{a_1}{b_1} + (1-\tilde{\lambda}  ) \frac{a_2}{b_2})  
\end{align*}
where $\tilde{\lambda}  =  \frac{\lambda b_1} {\lambda b_1 + (1-\lambda) b_2 } $. Hence, due to the convexity of $\psi$ 
\begin{align*}
& K(\lambda_1 a_1 +\lambda_2 a_a, \lambda_1 b_1 +\lambda_1 b_2) \\ & \leq  (\lambda b_1 + (1-\lambda )b_2)  \tilde{\lambda} \psi (\frac{a_1}{b_1}) + (\lambda b_1 + (1-\lambda )b_2)  (1 -\tilde{\lambda}) \psi (\frac{a_2}{b_2})   \\
& = \lambda b_1   \psi (\frac{a_1}{b_1}) + (1-\lambda )b_2  \psi (\frac{a_2}{b_2})  \\ 
& = \lambda K(a_1,b_1) + (1-\lambda) K (a_2,b_2)
\end{align*}

Given these definitions, we define the feasbile set of the convexified problem.

\begin{align} 
AD(\mu_0,\mu_T) = &\{ (m,\mu) \in \mathcal{M}([0,T] \times \Rd ; \R^m) \times C([0,T];\mathcal{P}_1(\Rd)),~ m << \mu; \\
  &  \int_{I }\int_{\Rd} [ \partial_t \phi + \partial_x\phi \cdot f(x,\frac{dm(x,t)}{d\mu_t(x)})]  d\mu_t(x)dt \nonumber \\
   & = \int_{\Rd}\phi(T,x)d\mu_T(x) - \int_{\Rd}\phi(0,x)d\mu_0(x), ~~ \forall \phi \in C^1_c(I \times \Rd)  \}
\end{align}
Note that, due to $f(x,u)$ being affine in $u$, the map $(m,\mu) \mapsto \int_{I }\int_{\Rd} [ \partial_t \phi + \partial_x\phi \cdot f(x,\frac{dm(x,t)}{d\mu_t(x)})] $ is linear in $(m,\mu)$. Hence, the set $C^{ac}(\mu_0,\mu_T)$ is a convex set. Given these definitions, we are interested in the following convex Benamou-Brenier problem.

\begin{align}
C^{conv}_{BB}(\mu_0,\mu_T) = \inf_{(m,\mu) \in AD(\mu_0,\mu_T)} \mathcal{F}(m,\mu)
\end{align} 
The following result can be stated for this problem.
\begin{theorem}
\label{thm:conbb}
Given Assumption \ref{asmp:costco} (Tonelli running cost) and \ref{asmp:costco} (vector-fields with sub-linear growth). Suppose $\mu_0, \mu_T \in \mathcal{P}_1(\Rd)$ are absolutely continuous with respect to the Lebesgue measure and compactly supported. Let $L(t,x,u) := s(x) + \psi(u)$ with $\psi$ convex, for all $(t,x,u) \in I \times \Rd \times U$ with $U = \R^d$. Suppose there exists a solution $(u,\mu)$ to the Benamou-Brenier formulation \eqref{eq:main_optimbb} such that $\mu_t$ is absolutely continuous with respect to the Lebesgue measure for Lebesgue almost every $t \in [0,T]$. Then we have
\[C_{BB}(\mu_0,\mu_T) = C^{conv}_{BB}(\mu_0,\mu_T) \]
\end{theorem}
\begin{proof}
Suppose $(m,\mu) \in AD(\mu_0,\mu_T)$ is such that  $\mathcal{F}(m,\mu) < \infty$ then it is easy to see that there exists $(u,\mu)$ that $\int_I \int_{\Rd } L(t,x,u(t,x)) d\mu_t(x)dt  = \mathcal{F}(m,\mu)$, by setting $u(t,x) = \frac{dm(x,t)}{d\mu_t(x)}$. This implies that $C^{conv}_{BB}(\mu_0,\mu_T) \geq C_{BB}(\mu_0,\mu_T)$. Now, suppose $(u,\mu)$ is such that it solves the Benamou-Brenier problem \eqref{eq:main_optimbb}. Then clearly, by setting $m= u \mu$, we get that $(m,\mu) \in AD(\mu_0,\mu_T)$ and by assumption we can choose $\mu_t$ to be absolutely continuous for almost every $t \in [0,T]$ we get that $\mathcal{F}(m,\mu) = C_{BB}(\mu_0,\mu_T)$.
\end{proof}

Next, we can apply this result to some concrete cases for which the curve $\mu_t$ is known to be absolutely continuous due to results from the corresponding Monge problem \cite{agrachev2009optimal,figalli2010mass}. In order to state the result we will need some additional assumptions on the system. We define the \textit{endpoint map} $\mathcal{E} : L^2(0,T;\R^m)\rightarrow \R^d$ as
\begin{equation}
\mathcal{E}: u \mapsto \omega(T),
\end{equation}
where $\omega(T)$ is the terminal point of the trajectory $\omega$ associated to the control $u$.

\begin{definition}
A trajectory-control pair $(\omega, u) \in \Omega$ is said to be a \textbf{singular curve} if the differential of the endpoint map at $u(\cdot)$,
\begin{equation}
D\mathcal{E}(u(\cdot)) : L^2([0,T]; \mathbb{R}^m) \to \Rd,
\end{equation}
is not surjective.
\end{definition}

An assumption we will need is that if the system is driftless then there exist no pair $(x,y) \in \Rd \times \Rd$, for which the minimizing trajectory pair $(\omega, u) \in \Omega$ that achieve the cost $c(x,y)$ are singular. An example of such a driftless system is if it is {\it $2$-generating}. We recall that a driftless system is said to be $2$-generating if the satisfies the controllability assumption \ref{asmp:ctbdrif} and 
$\cup_{i =0}^2 \mathcal{V}^i$ has rank $d$. That is, if ${\rm span} \{ g(x) \in \cup_{i =0}^2 \mathcal{V}^i \}$, where $\mathcal{V}^i$ is defined in Assumption \ref{asmp:ctbdrif}. Other examples of systems that do not admit singular minimizing curves include systems for which the admissible directions fall under the class of {\it fat} distributions. See \cite{figalli2010mass,rifford2014sub}. Another much larger class of distributions, known as {\it medium fat} are known {\it generically} not admit singular minimizers.

Given these definitions, we can state the following result for systems, where we recall that $D \subset \Rd \times \Rd$ is the set of points of the form $(x,x)$.

\begin{corollary}
Given \ref{asmp:costco} (vector-fields with sub-linear growth), let $L(t,x,u) = \sum_{i=1}^m|u_i|^2$ be the minimum energy cost, for all $(t,x,u) \in I \times \Rd \times U$ with $U = \R^d$. Suppose $\mu_0, \mu_T \in \mathcal{P}_1(\Rd)$ are absolutely continuous with respect to the Lebesgue measure and compactly supported. Additionally assume that the system is a controllable driftless system according to Assumption \ref{asmp:ctbdrif} for which there exists no $(x,y) \in \mathbb{R}^{d} \times \Rd \setminus (D \times D)$ admitting a singular minimizing curve and the corresponding cost function induces a complete metric on $\Rd$. Then we have
\[C_{BB}(\mu_0,\mu_T) = C^{conv}_{BB}(\mu_0,\mu_T) \]
\end{corollary}
\begin{proof}
Under the assumptions made on the absence of singular minimizers, from \cite{figalli2010mass}[Theorem 5.9], we know that the the cost function is locally semiconcave on $\Rd \times \Rd \setminus D$. In this case, it is known due to \cite{figalli2010mass}[Theorem 6.3],  the Kantorovich problem \eqref{eq:kant} is equivalent to the Monge problem
\[\inf_{T:\Rd \rightarrow \Rd; \\ T_{\#}\mu_0 =\mu_T } = \int_{\Rd \times \Rd} c(x,T(x))d\mu_0(x)\]
Moreover, for $\mu_0$ almost every $x \in \Rd$ there exists an unique admissible pair $(\omega_x,\alpha_x) \in \Omega$ such that $T(x) = \omega_x(T)$. Define $S:\Rd \times \Rd \rightarrow \Omega$ by,
$S(x,T(x))= (\omega_x,\alpha_x) $ for $\mu_0$ almost every $x \in \Rd$. This map is defined $\gamma$ almost every where on $\Rd \times \Rd$ for $\gamma = (id,T)_{\#}\mu_0$ where $id:\Rd \rightarrow \Rd$ is the identity map. It is known that  $\mu_t :=[\omega_x(t)]_{\#} \mu_0$ is absolutely continuous with respect to the Lebesgue measure for almost every $t \in [0,T]$,by \cite{figalli2010mass}[Theorem 3.5]. Since the flow $x \mapsto \omega_x(t)$ injective for $\mu_0$ almost every $x \in \Rd$ and for almost every $t \in [0,T]$, we can define $u(t,x) := \alpha_x(t)$ for $\mu_t$ almost every $x \in \Rd$ for almost every $t \in [0,T]$. Clearly, 
\begin{align}
\int_\Rd c(x,T(x))d\mu_0(x) &= \int_{\Rd}\int_{I} L(t,\omega_x(t),u(t,\omega_x(t))dt d\mu_0(x) \\ 
& = \int_{\Rd}\int_{I}  
 L(t,x,u(t,x))dt d\mu_t(x) 
\end{align}
Following as in proof of Proposition \ref{prop:marg2yng}, we can establish that the couple $(u,\mu)$ solve the continuity equation. Since the Benamou-Brenier problem and the Kantorovich problem achieve the same minimum by Theorem \ref{thm:equivkbb}, we conclude the result. 
\end{proof}

Next, we can state the result for the special case of Linear time invariant systems with linear quadratic cost. This result generalizes a result due to \cite{chen2016optimal}.

\begin{corollary}
\textbf{(Linear quadratic cost)}
Let $\mu_0, \mu_T \in \mathcal{P}_2(\Rd)$. Suppose the control system \eqref{eq:ctrsyspre} is a controllable LTI  according to Assumption \eqref{asmp:ctbllin}. Let the running cost be given by $L(t,x,u) = \langle x,Qx  \rangle_2 +\langle u,Ru  \rangle_2$ for some postive semidefinite matrix $Q \in \R^{d  \times d}$ and positive definite matrix $R \in \R^{n \times n}$, for all $(t,x,u) \in I \times \Rd \times U$ with $U = \R^d$. Then the Benamou-Brenier problem is a solution to the convexified Benamou-Brenier problem.
\end{corollary}
\begin{proof}
The cost $c(x,y)$ is Lipschitz due to the representation derived in \cite{hindawi2011mass}. Then, by first invoking the existence of Monge problem for this case due to \cite{agrachev2009optimal}[Theorem 4.1], the proof follows verbatim as in the proof of the previous corollary.
\end{proof}

\subsection{Controllability and Transport with Uniqueness of Continuity Equation Solution}

In general, the controls that transport an initial measure $\mu_0$ to the terminal value $\mu_T$ are potentially very irregular. Due to this lack of smoothness, one cannot usually guarantee that the corresponding continuity equation has a unique weak solution. In this section, we investigate under what conditions the optimal transport control laws are regular enough so that the corresponding continuity equation has a unique solution. Like in the previous section the existence of solutions to the Monge problem, as established by \cite{agrachev2009optimal,figalli2010mass} plays an important role. 

One can think of this as a stronger version of a controllability result for the continuity equation that has been proved in \cite{duprez2019approximate}, where controllability has been proved with Borel controls for the fully actuated case ($m=d$ and $g_i(x)$ are coordinate vector fields). However, the control in \cite{duprez2019approximate} that achieves (exact) controllability can cause non-unique solutions. In contrast, in the following result, using the ideas based on viscosity solutions of Hamilton Jacobi Bellman equation and the corresponding optimal synthesis problem \cite{cannarsa2004semiconcave}, we show that the corresponding continuity equation cannot develop non-unique solutions.

\begin{theorem}
\label{thm:uniquesol1}
Given \ref{asmp:costco} (vector-fields with sub-linear growth), Suppose $\mu_0, \mu_T \in \mathcal{P}_1(\Rd)$ are absolutely continuous with respect to the Lebesgue measure and compactly supported. Additionally assume that the system is a controllable driftless system according to Assumption \ref{asmp:ctbdrif} for which there exists no $(x,y) \in \Rd \times \Rd \setminus D \times D$ admitting a singular minimizing curve and the corresponding cost function induces a complete metric on $\Rd$.  Suppose $(u,\mu)$ is the solution of the Benamou Brenier problem for $L(t,x,u) = \frac{1}{2}\sum_{i=1}^m|u_i|^2$, for all $(t,x,u) \in I \times \Rd \times U$ with $U = \R^d$. 

Then for the feedback control $u$, the weak solution $\hat{\mu} \in C([0,T];\mathcal{P}_1(\Rd)$ to the continuity equation is unique.
\end{theorem}

\begin{proof}
Consider the dual Kantorovich problem
\begin{equation}
\label{eq:kant_dual}
C_{\rm kan,du}=\sup_{(\phi,\psi) \in K} \int \phi(x) d\mu_0(x) - \int \psi(x) d\mu_T(x) 
\end{equation}
where $K : = \lbrace (\phi, \psi) \in L^1(\mu_0) \times L^1(\mu_T) ; \phi(x) - \psi(y) \leq c(x,y), ~ \forall (x,y) \in \Rd \times \Rd \rbrace$.
 From \cite{figalli2010mass}[Theorem 3.2], the primal problem \eqref{eq:kant} and the above dual problem have the same solution $C_{\rm Kan, du} =C_{\rm kan}$. The cost function is continuous. Hence, the pair that achieve the supremum in the dual problem \eqref{eq:kant_dual} can be taken to be of the form $(f,f^c)$ functions $f:\Rd \rightarrow \R$ and $f^c:\Rd \rightarrow \R$, where $f^c$ are continuous functions given by 
\[f^c(y) = \inf_{x \in \mathbb{R}^d} c(x,y) + f(x)\]
We define the value function $V: [0,T] \times \Rd \rightarrow \R$,
\[V(y,t) = \inf_{(\gamma,u) \in \Omega; \gamma(t) = y} \int_0^t L(x(t),u(t))dt + f(\gamma(0))\]
and hence $V(y,T) = f^c(y)$ for all $y \in \Rd$.
Now, it is known from \cite{bardi1997optimal}[Proposition 3.5] (See also \cite{cannarsa2004semiconcave}) that the value function is a viscosity solution of the Hamilton Jacobi Bellman (HJB) equation,
\begin{align}
\partial_t V + H(y,\partial_x V)= 0  \\
V(x,0) = f(x)
\end{align}
where $H(y,p) = \max_{u} [p^Tf(x,u)+ \sum_{i=1}^m\frac{|u_i|^2}{2}] $. Note that theorem in \cite{bardi1997optimal,cannarsa2004semiconcave} make assumptions on the compactness of the controls unlike the hypothesis of the current theorem. However, for initial conditions from a compact set, a global bound on the $L^{\infty}$ norm of the optimal control can be derived to apply the result of \cite{cannarsa2004semiconcave} using the fact that the optimal controls are regular due to the non-existence of non-trivial singular minimizing curves (\cite{cannarsa2008semiconcavity}[Proposition 3.3]). Given the assumption that the system is driftless, the HJB equation reduces to
\begin{align}
\partial_t V + \frac{1}{2}\sum_{i=1}^m | g^T_i \partial_x V|^2= 0  \\
V(x,0) = f(x)
\end{align}
Let $T:\mathbb{R}^d \times \mathbb{R}^d \rightarrow \Rd$ be the solution to the corresponding Monge problem. Define the set, $M = \{x \in \Rd; T(x) \neq x \}$ and $S = M^c$. where the superscript $c$ denotes the complement of $M$. The set $M$ and $S$ is known as the moving set and the static set of the transport map. It is known that the set $M$ is open and $f$ is locally semiconcave on a neighborhood of $M \cap {\rm sup} ~\mu_0$ by \cite{figalli2010mass}[Theorem 3.2]. From this it also follows that $V$ is locally semiconcave on $[0,T] \times \tilde{M}$ where $\tilde{M}$ is an open neighborhood of $M \cap {\rm sup} ~\mu_0$. See \cite{cannarsa2004semiconcave}[Theorem 7.4.11]. While the theorem is stated for local semiconcavity on $[0,T] \times \Rd$ when $f$ is locally semiconcave on $\Rd$, since $\tilde{M}$ is open, the result \cite{cannarsa2004semiconcave}[Theorem 7.4.11] can be proved similarly for local semiconcavity of $V$ on $[0,T]\times \tilde{M}$ , when $f$ is only locally semiconcave on $\tilde{M}$.

According to \cite{cannarsa2004semiconcave}[Theorem 7.4.17], for $\mu_0$ almost every $x \in M$ the optimal trajectory is given by the solution of the Pontryagin system,

\begin{align}
\label{eq:pmpsys}
   &  \dot{\omega_x}(t) = \sum_{i=1}^m  u^i_x(t) g_i(x)  \\
   &   \dot{p}_x(t) = -\hat{H}_{\omega} (\omega_x(t),p(t)) \\
   &  \omega_x(0) = x \\
   &   p_x(0) = \partial_x f(x)      \\
   &   u_x(t) = \argmax_{u} [p^T(t)f(x,u) + \frac{u^2}{2}] =  -p_x(t)^TG(x)
\end{align}

where $G(x) : = [g_1(x),...g_m(x)]$ and $\hat{H}(x,p): = p^Tf(x,u + \sum_{i=1}^m \frac{|u_i|^2}{2})$ is the unmaximized Hamiltonian. Note that since $f$ is locally semiconcave on $\tilde{M}$, $\partial_x f(x)$  exists $\mu_0$ almost every $x \in M$. 
By the relation between Pontryagin Maximum Principle (PMP) and HJB equations, $p_x(t) = \partial_y V(t,\omega_x(t))$ for almost every $t $ and almost $\mu_0$ almost every $x \in M$ since $V$ is locally semiconcave on $\tilde{M}$ by \cite{cannarsa2004semiconcave}[Theorem 7.4.16]. This implies $V(t,y)$ is differentiable for $\mu_t$ every $y \in M$ and almost every $t \in [0,T]$. For $i=1,...,m$, let us define the control 
\begin{equation}
u_i(x,t) = 
\begin{cases}
-g^T_i(x) \partial_x V(x,t) ,~{\rm if}~ x \in M \cap {\rm sup}~ \mu_t \\
0,~{\rm otherwise}
\end{cases}
\end{equation}
From this construction we can see that for $\mu_0$ almost every $x \in \Rd$, the state-trajectories $\omega_x(t)$ satisfy the differential equation,
\begin{align}
\label{eq:optimalODE}
 & \dot{\omega_x}(t) = \sum_{i=1}^m u_i(t,\omega_x(t)) g_i(\omega_x(t)) \\
  &  \omega_x(0) = x 
\end{align}
We want to prove that the corresponding solution of the continuity equation is unique for $\mu_0$ almost every $x \in \Rd$. Toward this end we use a result on relation between uniqueness of solutions of the ODEs and uniqueness of weak solutions of the continuity equation \cite{ambrosio2008transport}. Additionally, we will use the fact the for $\mu_0$ almost every $x \in \Rd$ there exists a unique optimal trajectory. 

Suppose there is a solution $\gamma(t)$ for the ODE \eqref{eq:optimalODE} with $\gamma(0) = x$ and $\gamma \neq \omega_x$ for any $x$. We consider two cases. First, suppose $x \in S$. That is, the initial condition starts at the static set. Since $\omega_x \neq \gamma$, there must exist an interval $\tau \in [0,T]$ and $\epsilon>0$, $\gamma(\tau) = x$ for almost every $[0,\tau]$ and $\gamma(t) \in M$ for almost every $t \in (\tau +\epsilon]$. Clearly, $V(t,\gamma(t))$ is differentiable for almost every $t \in [0,T]$ since $\gamma$ is Lipschitz. Moreover, $\frac{d}{dt}V(t,\gamma(t)) =0$ for $t \in [0,\tau]$ since $\gamma(t)$ remains equal to $x$. Additionally, we know that $V(t,x)$ is differentiable at $x = \gamma(t)$ for almost every $t \in (\tau,\tau+\epsilon]$ since $\gamma$ satisfies the differential equation \eqref{eq:optimalODE}.
 
Since $V$ is locally semiconcave on $M$ from \cite{rifford2009stabilization}[Lemma 2.2] we can compute the time derivative of $V(t,y)$ along $y = \gamma(t)$,
\begin{align}
\sum_{i=1}^m\frac{|\alpha_i(t)|^2}{2}+\frac{d}{dt}V(t,\gamma(t)) &= \partial_tV(t,\gamma(t)) + H(\gamma(t),\partial_xV(t,\gamma(t))) = 0
\end{align}
 for almost every $t \in (\tau,\tau+\epsilon]$.
where $\alpha_i(t) = -g^T_i(\gamma(t)) \partial_x   V(t,\gamma(t))$.
Extending $\alpha_i(t)$ to the entire interval $[0,\tau+\epsilon]$ by setting $\alpha_i(t) = 0$ for all $ t\in [0, \tau]$ we see that,
\[V(\tau+\epsilon,\gamma(\tau+\epsilon)) = V(0,\gamma(0)) + \frac{1}{2}\int_0^T \sum_{i=1}^m|\alpha_i(t)|^2 dt\]
 This implies that $\gamma(t)$ is an optimal trajectory connecting $\gamma(0)$ and $\gamma(\tau+\epsilon)$. However, note that the control $\hat{\alpha}:[0,\tau+\varepsilon] \rightarrow \R^m$ defined by 
 \[\hat{\alpha}(t) = \alpha(\tau+\frac{\varepsilon t}{  \tau+\varepsilon}), ~ \forall t \in [0,\tau+\varepsilon] \]
 also transfers the control system \eqref{eq:ctrsyspre} from $\gamma(0) $ to $\gamma(\tau+\varepsilon)$ over the time interval $[0,\tau +\varepsilon]$ since the control system is driftless. However, $\int_0^{\tau +\varepsilon} |\alpha(t)|^2dt  > \frac{\varepsilon}{\tau +\varepsilon} \int_0^{\tau +\varepsilon} |\alpha(t)|^2 dt = \int_0^{\tau +\varepsilon} |\hat{\alpha}(t)|^2 dt $. This implies that $\gamma(t)$ cannot be an optimal trajectory and giving us a contradiction. Therefore, the solutions of \eqref{eq:optimalODE} for $\mu_0$ almost every $x \in S$ are unique. Now consider the case when $x \in M$. Similarly, any other trajectory $\gamma(t)$ is optimal. However, $\partial_xf$ is differentiable $\mu_0$ almost every $M$. Hence, the optimal trajectories as a solution to the Pontryagin system \eqref{eq:pmpsys} must be unique since the corresponding flow of the Hamiltonian systems is Lipschitz. This enables us to conclude that the solutions of \eqref{eq:optimalODE} are unique $\mu_0$ almost every $\Rd$.
 
 Now, we can apply \cite{ambrosio2008transport}[Theorem 3.1] to conclude that weak solutions of the corresponding continuity equation are unique. 
 \end{proof} 

A straightforward corollary of the above result is the following result.

 \begin{corollary}
Given \ref{asmp:costco} (vector-fields with sub-linear growth), Suppose $\mu_0, \mu_T \in \mathcal{P}_1(\Rd)$ are absolutely continuous with respect to the Lebesgue measure and compactly supported. Additionally, assume that the system is a controllable driftless system according to Assumption \ref{asmp:ctbdrif} that is $2-$generating and the corresponding cost function induces a complete metric on $\Rd$.  Suppose $(u,\mu)$ is the solution of the Benamou Brenier problem for $L(t,x,u) = \frac{1}{2}\sum_{i=1}^m|u_i|^2$, for all $(t,x,u) \in I \times \Rd \times U$ with $U = \R^d$. 

Then for the feedback control $u$, the weak solution $\hat{\mu} \in C([0,T];\mathcal{P}_1(\Rd)$ to the continuity equation is unique.
\end{corollary}

Next, we state a similar result for the linear quadratic case. In this situation, we can guarantee that the control laws have some additional regularity. Particularly, the control laws are locally in of bounded variation for almost every time. We recall that a function $f \in L^1(\Omega)$ is in $BV_{\rm loc}(\Rd)$ if its distributional derivative is a finite Radon measure.

\begin{theorem}
\label{thm:uniquesol2}
Let $\mu_0, \mu_T \in \mathcal{P}_2(\Rd)$. Suppose the control system \eqref{eq:ctrsyspre} is a controllable LTI  according to Assumption \eqref{asmp:ctbllin}. Let the running cost be given by $L(t,x,u) = \langle x,Qx  \rangle_2 +\langle u,Ru  \rangle_2$ for some postive semidefinite matrix $Q \in \R^{d  \times d}$ and positive definite matrix $R \in \R^{n \times n}$.  Suppose $(u,\mu)$ is the solution of the Benamou Brenier problem.  Then  $u(t,\cdot)$ is in $BV_{\rm loc}(\Rd)$ for almost every $ t \in [0,T]$. 

Moreover, for the feedback control $u$, the weak solution $\hat{\mu} \in C([0,T];\mathcal{P}_1(\Rd)$ to the continuity equation is unique.
\end{theorem}

\begin{proof}
The proof of this case follows exactly as in the previous theorem. The only difference is the extra regularity of the control law. Toward this end, we note that the cost function $c(x,y) $ is locally semiconcave due the representation $(c(x,y) = \langle x,Dx\rangle_2 - \langle x,Ey\rangle_2+\langle y,Fy \rangle_2 )$ derived in \cite{hindawi2011mass}. This implies that the optimal solutions of the dual Kantorovich problem are also locally semiconcave everywhere.  Hence, the value function defined by
\[V(t,y) = \inf_{(\gamma,u) \in \Omega; \gamma(t) = y} \int_0^t L(x(t),u(t))dt + f(\gamma(0))\]
also inherits the local semiconcavity of $c(x,y)$ due to the results of \cite{cannarsa2004semiconcave}[Theorem 7.4.11]. Therefore, $\partial_xV(t,x)$ is defined $\mu_t$ almost everywhere and $\partial_xV(t,x)$ is in $BV_{\rm loc}(\Rd;\Rd)$, as $V$ is locally a sum of a smooth function and a concave function, for almost every $ t \in [0,T]$. Concave functions are twice differentiable by Aleksandorov's theorem (\cite{evans2018measure}[Theorem 6.9]). Therefore, the optimal control laws in this case given by
\[ u_i(t,x) = -R^{-1}b_i^T \partial_xV(t,x),\]
 are also in $ BV_{\rm loc}(\Rd)$ for almost every $t \in [0,T]$.
\end{proof}

\section{Acknowledgement}

The author thanks Matt Jacobs and Emmanuel Tr\'{e}lat for helpful comments and
suggestions regarding properties of continuity equation and HJB equations.

\bibliographystyle{plain}
\bibliography{ref}
\end{document}